\documentclass[12pt,reqno]{amsart}
\usepackage{amsmath,amsthm,amsfonts,amssymb,times}
\usepackage[mathscr]{eucal}
\usepackage{verbatim}
\usepackage{url}
\setbox0=\hbox{$+$}
\newdimen\plusheight
\plusheight=\ht0
\def\+{\;\lower\plusheight\hbox{$+$}\;}

\setbox0=\hbox{$-$}
\newdimen\minusheight
\minusheight=\ht0
\def\-{\;\lower\minusheight\hbox{$-$}\;}

\setbox0=\hbox{$\cdots$}
\newdimen\cdotsheight
\cdotsheight=\plusheight
\def\cds{\lower\cdotsheight\hbox{$\cdots$}}

\makeatletter
\def\leqalignno#1{\displ@y \tabskip\z@ plus\@ne fil
  \halign to\displaywidth{\hfil$\@lign\displaystyle{##}$\tabskip\z@skip
    &$\@lign\displaystyle{{}##}$\hfil\tabskip\z@ plus\@ne fil
    &\kern-\displaywidth\rlap{$\@lign\hbox{\rm##}$}\tabskip\displaywidth\crcr
    #1\crcr}}
\makeatother

\newcommand{\eb}{\begin{equation}}
\newcommand{\ee}{\end{equation}}

\newcommand{\df}{\dfrac}
\newcommand{\tf}{\tfrac}

\renewcommand{\Re}{\operatorname{Re}}

\renewcommand{\l}{\lambda}

\renewcommand{\Re}{\textup{Re}}

\renewcommand{\(}{\left\(}
\renewcommand{\)}{\right\)}
\renewcommand{\[}{\left\[}
\renewcommand{\]}{\right\]}

\numberwithin{equation}{section}
 \theoremstyle{plain}
\newtheorem{theorem}{Theorem}[section]
\newtheorem{lemma}[theorem]{Lemma}

\newtheorem{remark}[theorem]{Remark}

\allowdisplaybreaks

\begin{document}
\title[Arithmetical identities]
{A Class of Identities Associated with Dirichlet Series Satisfying Hecke's Functional Equation}
\author{Bruce C.~Berndt, Atul Dixit, Rajat Gupta, Alexandru Zaharescu}
\address{Department of Mathematics, University of Illinois, 1409 West Green
Street, Urbana, IL 61801, USA} \email{berndt@illinois.edu}
\address{Department of Mathematics, Indian Institute of Technology Gandhinagar, Palaj, Gandhinagar 382355, Gujarat, India}\email{adixit@iitgn.ac.in}
\address{Department of Mathematics, Indian Institute of Technology Gandhinagar, Palaj, Gandhinagar 382355, Gujarat, India}\email{rajat\_gupta@iitgn.ac.in}
\address{Department of Mathematics, University of Illinois, 1409 West Green
Street, Urbana, IL 61801, USA; Institute of Mathematics of the Romanian
Academy, P.O.~Box 1-764, Bucharest RO-70700, Romania}
\email{zaharesc@illinois.edu}
\maketitle

\begin{abstract}
    We consider two sequences $a(n)$ and $b(n)$, $1\leq n<\infty$, generated by Dirichlet series of the forms
    $$\sum_{n=1}^{\infty}\df{a(n)}{\lambda_n^{s}}\qquad\text{and}\qquad
    \sum_{n=1}^{\infty}\df{b(n)}{\mu_n^{s}},$$
        satisfying a familiar functional equation involving the gamma function $\Gamma(s)$. A general identity is established.  Appearing on one side is an infinite series involving $a(n)$ and modified Bessel functions $K_{\nu}$, wherein on the other side is an infinite series involving $b(n)$ that is an analogue of the Hurwitz zeta function. Seven special cases, including $a(n)=\tau(n)$ and $a(n)=r_k(n)$, are examined, where $\tau(n)$ is Ramanujan's arithmetical function and $r_k(n)$ denotes the number of representations of $n$ as a sum of $k$ squares.  Most of the six special cases appear to be new.  
    \end{abstract}

\section{Introduction}

Our goal is to establish a new set of identities involving arithmetical functions whose generating functions are Dirichlet series satisfying Hecke's functional equation.  Our general theorem involves the modified Bessel function $K_{\nu}(z)$.  Even in the special case \cite[pp.~79, 80, no.~(13)]{watson}
\begin{equation}\label{K}
K_{1/2}(z)=\sqrt{\df{\pi}{2z}}e^{-z},
\end{equation}
most special cases are new.

We employ the setting of K.~Chandrasekharan and R.~Narasimhan from their paper \cite{cn1}. Throughout our paper, $\sigma=\Re(s)$.
Let $a(n)$ and $b(n)$, $1\leq n<\infty$, be two sequences of complex numbers, not identically 0.  Set
\begin{equation}\label{18}
\varphi(s):=\sum_{n=1}^{\infty}\df{a(n)}{\lambda_n^s}, \quad \sigma>\sigma_a; \qquad
\psi(s):=\sum_{n=1}^{\infty}\df{b(n)}{\mu_n^s}, \quad \sigma>\sigma_a^*,
\end{equation}
where  $\{\l_n\}$ and $\{\mu_n\}$ are two sequences of positive numbers, each tending to $\infty$, and $\sigma_a$ and $\sigma_a^*$ are the (finite) abscissae of absolute convergence for $\varphi(s)$ and $\psi(s)$, respectively.  We assume that $\varphi(s)$ and $\psi(s)$ have analytic continuations into the entire complex plane $\mathbb{C}$ and are analytic on $\mathbb{C}$ except for a finite set $\bf{S}$ of poles.
Suppose that for some $\delta>0$, $\varphi(s)$ and $\psi(s)$  satisfy a functional equation of the form
\begin{equation}\label{19}
\chi(s):=(2\pi)^{-s}\Gamma(s)\varphi(s)=(2\pi)^{s-\delta}\Gamma(\delta-s)\psi(\delta-s).
\end{equation}
Chandrasekharan and Narasimhan show that the functional equation \eqref{19} is equivalent to the following two identities \cite[p.~6, Lemmas 4, 5]{cn1} the first of which is due to Bochner \cite{bochner}.

\begin{theorem}\label{thmmodular} The functional equation \eqref{18} is equivalent to the `modular' relation
\begin{equation*}
\sum_{n=1}^{\infty}a(n)e^{-\lambda_n x}=\left(\df{2\pi}{x}\right)^{\delta}\sum_{n=1}^{\infty}b(n)e^{-4\pi^2\mu_n/x}+P(x), \qquad \Re(x)>0,
\end{equation*}
where
\begin{equation*}
P(x):=\frac{1}{2\pi i}\int_{\mathcal{C}}(2\pi)^z\chi(z)x^{-z}dz,
\end{equation*}
where $\mathcal{C}$ is a curve or curves encircling all of $\bf{S}$.
\end{theorem}

\begin{theorem}\label{thmriesz} Let $J_{\nu}(z)$ denote the ordinary Bessel function of order $\nu$.  Let $x>0$ and $\rho>2\sigma_a^*-\delta-\frac12$.  Then the functional equation \eqref{18} is equivalent to the identity 
\begin{gather}
\df{1}{\Gamma(\rho+1)}{\sum_{\lambda_n\leq x}}^{\prime}a(n)(x-\lambda_n)^{\rho}=
\left(\df{1}{2\pi}\right)^{\rho}
\sum_{n=1}^{\infty}b(n)\left(\df{x}{\mu_n}\right)^{(\delta+\rho)/2}J_{\delta+\rho}(4\pi\sqrt{\mu_n x})+Q_{\rho}(x),\label{21}
\end{gather}
where the prime $\prime$ on the summation sign on the left side indicates that if $\rho=0$ and $x\in\{\lambda_n\}$, then only $\tf12 a(x)$ is counted.  Furthermore,  
 $Q_{\rho}(x)$ is  defined by
\begin{equation}\label{22}
Q_{\rho}(x):=\df{1}{2\pi i}\int_{\mathcal{C}}\df{\chi(z)(2\pi)^zx^{z+\rho}}{\Gamma(\rho+1+z)}dz,
\end{equation}
where $\mathcal{C}$ is a curve or curves encircling  $\bf{S}$.
\end{theorem}

The restriction $\rho>2\sigma_a^*-\delta-\frac12$ can be replaced by $\rho>2\sigma_a^*-\delta-\frac32$ under certain conditions given in \cite[p.~14, Theorem III]{cn1}. Because we later use analytic continuation, this extension is not important for us here.

Theorem \ref{thmmodular} is not explicitly used in the sequel.  However, Theorem \ref{thmriesz} is the key to our main theorem, Theorem \ref{thm2}.

We conclude our paper with seven examples, including the following arithmetical functions: $r_k(n)$, the number of representations of $n$ as a sum of $k$ squares;  $\sigma_k(n)$, the sum of the $k$ powers of the divisors of $n$; Ramanujan's arithmetical function $\tau(n)$;  $\chi(n)$, a primitive character;  and $F(n)$, the number of integral ideals of norm $n$ in an imaginary quadratic number field.   The identity involving $r_k(n)$ is known (but not well known).  

\section{Preliminaries}

We refer readers to G.~N.~Watson's classical treatise for the definitions of the Bessel functions $J_{\nu}(z)$ and $K_{\nu}(z)$ \cite[pp.~15, 78]{watson}.  The following lemmas will be used in the sequel.

\begin{lemma}\label{asymptotic}\cite[pp.~199, 202]{watson} Let $J_{\nu}(x)$ denote the ordinary Bessel function of order $\nu$, and let $K_{\nu}(x)$ denote the modified Bessel function of order $\nu$. As $z\to\infty$,
\begin{align}
J_{\nu}(z)=&\left(\df{2}{\pi z}\right)^{1/2}\left\{\cos\left(z-\frac12\nu\pi-\frac14 \pi\right)+O\left(\df{1}{z}\right)\right\},\label{asymptotic1}\\
K_{\nu}(z)=&\left(\df{2}{\pi z}\right)^{1/2}e^{-z}\left\{1+O\left(\df{1}{z}\right)\right\}.\label{asymptotic2}
\end{align}
\end{lemma}

\begin{lemma}\label{watson2} \cite[p.~79]{watson} For each non-negative integer $m$ and arbitrary $\nu$,
\begin{equation}\label{29}
\left(\df{d}{zdz}\right)^m\{z^{\nu}K_{\nu}(z)\}=(-1)^mz^{\nu-m}K_{\nu-m}(z).
\end{equation}
\end{lemma}

\begin{lemma}\label{limit} \cite[p.~329]{bdkz} For $\Re(\nu)>0$,
\begin{equation*}
\lim_{z\to 0}z^{\nu}K_{\nu}(z)=2^{\nu-1}\Gamma(\nu).
\end{equation*}
\end{lemma}

\begin{lemma}\label{watson} \cite[Equation (2), p.~410]{watson}
  Assume that $\Re(\nu)+1>|\Re(\mu)|$ and that $a,b>0$.  Then
\begin{equation}\label{23}
\int_0^{\infty}t^{\mu+\nu+1}K_{\mu}(at)J_{\nu}(bt)dt=\df{(2a)^{\mu}(2b)^{\nu}\Gamma(\mu+\nu+1)}{(a^2+b^2)^{\mu+\nu+1}}.
\end{equation}
\end{lemma}

\begin{lemma}\label{Kintegral} \cite[p.~708, no.~16]{gr} For $\Re(\mu+1\pm \nu)>0$ and $\Re(a)>0$,
\begin{equation*}
\int_0^{\infty}x^{\mu}K_{\nu}(ax)dx=2^{\mu-1}a^{-\mu-1}
\Gamma\left(\df{1+\mu+\nu}{2}\right)\Gamma\left(\df{1+\mu-\nu}{2}\right).
\end{equation*}
\end{lemma}

\section{The Primary Theorem}

\begin{theorem}\label{thm2}  Assume that $\Re(\nu)> 0$ and $\Re(s)>0$.  Also assume that  $\delta+\rho+\nu+1>\sigma_a^*>0$.  Assume that the integral on the right side below converges absolutely.
Then,
\begin{gather}
\df{1}{\Gamma(\rho+1)}\sum_{n=1}^{\infty}a(n)\int_{\lambda_n}^{\infty}
(x-\lambda_n)^{\rho}x^{\nu/2}K_{\nu}(s\sqrt{x})dx\notag\\
=2^{3\delta+2\rho+\nu+1}s^{\nu}\pi^{\delta}\Gamma(\delta+\rho+\nu+1)
\sum_{n=1}^{\infty}\df{b(n)}{(s^2+16\pi^2\mu_n)^{\delta+\rho+\nu+1}}\notag\\
+\int_0^{\infty}Q_{\rho}(x)x^{\nu/2}K_{\nu}(s\sqrt{x})dx.\label{28}
\end{gather}\end{theorem}

\begin{proof} Assume that $\rho>2\sigma_a^*-\delta-\frac12$.  Multiply both sides of \eqref{21} by $x^{\nu/2}K_{\nu}(s\sqrt{x})$, where $s>0$, and integrate over $0\leq x<\infty$.  Let $F_1(\delta,\rho,\nu)$ denote the left-hand side and let $F_2(\delta,\rho,\nu)$ and $F_3(\delta,\rho,\nu)$ denote, in order, the two terms on the right-hand side that we so obtain.  

First, on the left-hand side of \eqref{21}, we readily find that, for $\Re(\nu)>0$,
\begin{equation}\label{26}
F_1(\delta,\rho,\nu)= \df{1}{\Gamma(\rho+1)}\sum_{n=1}^{\infty}a(n)\int_{\lambda_n}^{\infty}
(x-\lambda_n)^{\rho}x^{\nu/2}K_{\nu}(s\sqrt{x})dx.
\end{equation}

Second, using Lemma \ref{watson}, assuming that $\Re(\delta+\rho)+1>\Re(\nu)$, using Lemma \ref{asymptotic}, and inverting the order of summation and integration by absolute convergence, we have
\begin{align}
F_2(\delta, \rho, \nu)=(2\pi)^{-\rho}\sum_{n=1}^{\infty}b(n)\mu_n^{-(\delta+\rho)/2}I(\delta, \rho, \nu),    
\end{align}
where
\begin{align}\label{24}
I(\delta,\rho,\nu):=&\int_0^{\infty}x^{(\delta+\rho+\nu)/2}J_{\delta+\rho}(4\pi\sqrt{\mu_n x})K_{\nu}(s\sqrt{x})dx\notag\\
=&2\int_0^{\infty}t^{\delta+\rho+\nu+1}J_{\delta+\rho}(4\pi\sqrt{\mu_n}\,t)K_{\nu}(st)dt\notag\\
=&2\df{(2s)^{\nu}(8\pi\sqrt{\mu_n})^{\delta+\rho}\Gamma(\delta+\rho+\nu+1)}
{(s^2+16\pi^2\mu_n)^{\delta+\rho+\nu+1}}.
\end{align}
Hence, with the use of \eqref{24} on the right side of \eqref{21}, after simplification, we obtain the sum
\begin{equation}\label{25}
F_2(\delta,\rho,\nu) =2^{3\delta+2\rho+\nu+1}s^{\nu}\pi^{\delta}\Gamma(\delta+\rho+\nu+1)
\sum_{n=1}^{\infty}\df{b(n)}{(s^2+16\pi^2\mu_n)^{\delta+\rho+\nu+1}}.
\end{equation}
Now use analytic continuation in $\rho,\nu$, and $s$ to conclude that \eqref{25} is valid  provided that  $\delta+\rho+\nu+1>\sigma_a^*$ and $\Re(s)>0$.

Lastly, for the remaining term on the right side of \eqref{21},  we find that
\begin{gather}\label{27}
F_3(\delta,\rho,\nu)= \int_0^{\infty}Q_{\rho}(x)x^{\nu/2}K_{\nu}(s\sqrt{x})dx,
\end{gather}
provided that $\Re(\nu)>0$ and $\Re(s)>0$, and that the integral above converges absolutely.

Bringing \eqref{26}, \eqref{25}, and \eqref{27} together, we complete the proof of Theorem \ref{thm2}.

\end{proof}

Theorem \ref{thm2} is a generalization of theorems proved by the first author \cite{paper1}, \cite{III}, \cite[p.~154, Equation (6.11)]{V}.  Recalling the definition of the Hurwitz zeta function, observe that the series on the right-hand side of \eqref{28} is an analogue of the Hurwitz zeta function.  Thus, \eqref{28} provides an analytic continuation for the series on the right side of \eqref{28}.  

\section{Special Cases: Non-negative integer values of $\rho$}

If we set $\rho=0$ in Theorem \ref{thm2} and employ Lemmas \ref{watson2} and \ref{asymptotic}, we find that the left-hand side of \eqref{28} is given by
\begin{align}
\sum_{n=1}^{\infty}a(n)\int_{\lambda_n}^{\infty}x^{\nu/2}K_{\nu}(s\sqrt{x})dx
=&\df{2}{s^{\nu+2}}\sum_{n=1}^{\infty}a(n)\int_{s\sqrt{\lambda_n}}^{\infty}t^{\nu+1}K_{\nu}(t)dt\notag\\
=&-\df{2}{s^{\nu+2}}\sum_{n=1}^{\infty}a(n)\int_{s\sqrt{\lambda_n}}^{\infty}
\df{d}{dt}\left\{t^{\nu+1}K_{\nu+1}(t)\right\}dt\notag\\
=&\df{2}{s^{\nu+2}}\sum_{n=1}^{\infty}a(n)(s\sqrt{\lambda_n})^{\nu+1}K_{\nu+1}(s\sqrt{\lambda_n})\notag\\
=&\df{2}{s}\sum_{n=1}^{\infty}a(n)\lambda_n^{(\nu+1)/2}K_{\nu+1}(s\sqrt{\lambda_n}).\label{30a}
\end{align}
Hence, we have established the following theorem.

\begin{theorem}\label{thm3}  Assume that $\Re(\nu)>0$ and $\Re(s)>0$.  Also assume that  $\delta+\nu+1>\sigma_a^*>0$.  Assume that the integral on the right side below converges absolutely.
Then,
\begin{align}
\df{2}{s}\sum_{n=1}^{\infty}a(n)\lambda_n^{(\nu+1)/2}K_{\nu+1}(s\sqrt{\lambda_n})
=&2^{3\delta+\nu+1}s^{\nu}\pi^{\delta}\Gamma(\delta+\nu+1)
\sum_{n=1}^{\infty}\df{b(n)}{(s^2+16\pi^2\mu_n)^{\delta+\nu+1}}\notag\\
&+\int_0^{\infty}Q_{0}(x)x^{\nu/2}K_{\nu}(s\sqrt{x})dx.\label{30}
\end{align}\end{theorem}

Return to Theorem \ref{thm2} and, for any $m\in\mathbb{N}$, let $\rho=m$. Apply the binomial theorem on the left-hand side.
In the integrand we obtain polynomials in $x$ of degree $k$, $0\leq k\leq m$. For the integral involving $x^k$, integrate by parts $k$ times with the aid of Lemma \ref{watson2}.  With the help  of the binomial theorem once again, simplify the double sum that arises.  For the integral on the right-hand side of \eqref{28}, integrate by parts $m$ times with the help of Lemma \ref{watson2}.  In conclusion, after simplification, we obtain  Theorem \ref{thm3} with $\nu$ replaced by $\nu+m$.

Chandrasekharan and Narasimhan \cite[p.~8, Lemma 6]{cn1} proved the following theorem, which is similar in appearance to Theorem \ref{thm2} in the special case $\nu=1/2$.
\begin{theorem} Let $\rho$ denote a non-negative integer, $\Re(s)>0$, and $\rho>2\sigma_a^*-\delta-\tf12$. Then
\begin{align*}
&\left(-\df{1}{s}\df{d}{ds}\right)^{\rho}\left\{\df{1}{s}\sum_{n=1}^{\infty}a(n)e^{-s\sqrt{\lambda_n}}\right\}\\
=&2^{3\delta+\rho}\Gamma(\delta+\rho+\tf12)\pi^{\delta-1/2}
\sum_{n=1}^{\infty}\df{b(n)}{(s^2+16\pi^2\mu_n)^{\delta+\rho+1/2}}+R_{\rho}(s),
\end{align*}
where
$$ R_{\rho}(s):=\df{1}{2\pi i}
\int_{\mathbb{C}}\df{\chi(z)(2\pi)^z\Gamma(2z+2\rho+1)2^{-\rho}s^{-2z-2\rho-1}}{\Gamma(z+\rho+1)}dz.$$
\end{theorem}

\section{Example 1: $r_k(n)$}

In the examples below we refer to calculations made  by Chandrasekharan and Narasimhan \cite{cn1} to illustrate Theorem \ref{thmriesz}.  In particular, we use a few of their determinations of $Q_{\rho}(x)$.

\bigskip

 Let $r_k(n)$ denote the number of representations of the positive integer $n$ as a sum of $k$ squares, where  $k\geq2$.  Then
$$\zeta_k(s):=\sum_{n=1}^{\infty}\df{r_k(n)}{n^s},\qquad \sigma>k/2,$$
satisfies the functional equation
\begin{equation}\label{functionalequation}
 \pi^{-s}\Gamma(s)\zeta_k(s)=\pi^{s-k/2}\Gamma(k/2-s)\zeta_k(k/2-s).
 \end{equation}
In the notation of \eqref{19}, 
\begin{equation*}
a(n)=b(n)=r_k(n), \qquad \delta=k/2, \qquad \text{and} \qquad \lambda_n=\mu_n=n/2.
\end{equation*}
 From the functional equation \eqref{functionalequation}, $\zeta_k(0)=-1$, and $\zeta_k(s)$ has a simple pole at $s=2k$ with residue $\pi^{k/2}/\Gamma(k/2)$. 
We now apply Theorem \ref{thm3}. First, from the preceding remarks, 
\begin{equation}\label{31}
Q_{0}(x)=-1+\df{(2\pi)^{k/2}x^{k/2}}{\Gamma(1+k/2)}.
\end{equation}
Second, we calculate the integral on the right side of \eqref{30}.  To that end,
\begin{align}\label{32}
I:=&\int_0^{\infty}\left(-1+\df{(2\pi)^{k/2}x^{k/2}}{\Gamma(1+k/2)}\right)x^{\nu/2}K_{\nu}(s\sqrt{x})dx\notag\\
=&2\int_0^{\infty}\left(-1+\df{(2\pi)^{k/2}t^{k}}{\Gamma(1+k/2)}\right)t^{\nu+1}K_{\nu}(st)dt\notag\\
=&:I_1+I_2.
\end{align}
Using Lemmas \ref{watson2}, \ref{asymptotic}, and \ref{limit} in order, we find that
\begin{align}\label{33}
I_1&=-2\int_0^{\infty}t^{\nu+1}K_{\nu}(st)dt\notag\\
&=-\df{2}{s^{\nu+2}}\int_0^{\infty}x^{\nu+1}K_{\nu}(x)dx\notag\\
&=\df{2}{s^{\nu+2}}\int_0^{\infty}\df{d}{dx}\left\{x^{\nu+1}K_{\nu+1}(x)\right\}dx\notag\\
&=-\df{2^{\nu+1}}{s^{\nu+2}}\Gamma(\nu+1).
\end{align}
 Secondly, using Lemma \ref{Kintegral}, we find that
\begin{align}
I_2=&\df{2(2\pi)^{k/2}}{\Gamma(1+k/2)}\int_0^{\infty}t^{k+\nu+1}K_{\nu}(st)dt\notag\\
=&\df{2(2\pi)^{k/2}}{s^{k+\nu+2}\Gamma(1+k/2)}\int_0^{\infty}x^{k+\nu+1}K_{\nu}(x)dx\notag\\
=&\df{2^{3k/2+\nu+1}\pi^{k/2}}{s^{k+\nu+2}}\Gamma(\nu+1+k/2).
\label{34}
\end{align}
Putting \eqref{33} and \eqref{34} in \eqref{32}, we conclude that
\begin{equation}\label{35}
I=-\df{2^{\nu+1}}{s^{\nu+2}}\Gamma(\nu+1)+\df{2^{3k/2+\nu+1}\pi^{k/2}}{s^{k+\nu+2}}\Gamma(\nu+1+k/2).
\end{equation}

Using \eqref{35} in Theorem \ref{thm3}, we deduce that
\begin{gather}
\df{2}{s}\sum_{n=1}^{\infty}r_k(n)\left(\df{n}{2}\right)^{(\nu+1)/2}K_{\nu+1}\left(s\sqrt{\df{n}{2}}\right)
+\df{2^{\nu+1}}{s^{\nu+2}}\Gamma(\nu+1)\notag\\
=\df{2^{3k/2+\nu+1}\pi^{k/2}}{s^{k+\nu+2}}\Gamma(\nu+1+k/2)
+2^{3k/2+\nu+1}s^{\nu}\pi^{k/2}\Gamma(\nu+1+k/2)\sum_{n=1}^{\infty}\df{r_k(n)}{(s^2+8\pi^2n)^{k/2+\nu+1}}.
\label{36}
\end{gather}
Now, let $s=2^{3/2}\pi\sqrt{\beta}.$ We therefore write \eqref{36} as
\begin{align}
&\df{1}{\pi\sqrt{2\beta}}\sum_{n=1}^{\infty}r_k(n)\left(\df{n}{2}\right)^{(\nu+1)/2}
K_{\nu+1}\left(2\pi\sqrt{n\beta}\right)
+\df{2^{\nu+1}}{(2^{3/2}\pi\sqrt{\beta})^{\nu+2}}\Gamma(\nu+1)\notag\\
=&
\df{2^{3k/2+\nu+1}\pi^{k/2}\Gamma(\nu+1+k/2)}{(2^{3/2}\pi\sqrt{\beta})^{k+\nu+2}}\notag\\
&+\df{2^{3k/2+\nu+1}(2^{3/2}\pi\sqrt{\beta})^{\nu}\pi^{k/2}\Gamma(\nu+1+k/2)}{(8\pi^2)^{k/2+\nu+1}}
\sum_{n=1}^{\infty}\df{r_k(n)}{(\beta+n)^{k/2+\nu+1}}.
\label{37}
\end{align}
Multiplying both sides of \eqref{37} by $2^{1+\nu/2}\pi\sqrt{\beta}$ and simplifying, we arrive at
\begin{gather}
\sum_{n=1}^{\infty}r_k(n)n^{(\nu+1)/2}
K_{\nu+1}\left(2\pi\sqrt{n\beta}\right)+\df{\Gamma(\nu+1)}{2\pi^{\nu+1}\beta^{(\nu+1)/2}}\notag\\
=\df{\beta^{(\nu+1)/2}\Gamma(\nu+1+k/2)}{2\pi^{k/2+\nu+1}\beta^{k/2+\nu+1}}
+\df{\beta^{(\nu+1)/2}\Gamma(\nu+1+k/2)}{2\pi^{k/2+\nu+1}}
\sum_{n=1}^{\infty}\df{r_k(n)}{(\beta+n)^{k/2+\nu+1}}.
\label{38}
\end{gather}
If we define $r_k(0)=1$ and formally use Lemma \ref{limit}, we find that
\begin{equation*}
\lim_{n\to 0}r_k(0)n^{(\nu+1)/2}K_{\nu}(2\pi\sqrt{n\beta})=\df{\Gamma(\nu+1)}{2\pi^{\nu+1}\beta^{(\nu+1)/2}}.
\end{equation*}
Hence, we see that we can put \eqref{38} in the form
\begin{equation}\label{39}
\sum_{n=0}^{\infty}r_k(n)n^{(\nu+1)/2}
K_{\nu+1}\left(2\pi\sqrt{n\beta}\right)=
\df{\beta^{(\nu+1)/2}\Gamma(\nu+1+k/2)}{2\pi^{k/2+\nu+1}}
\sum_{n=0}^{\infty}\df{r_k(n)}{(\beta+n)^{k/2+\nu+1}}.
\end{equation}
From the fact that $r_k(n)=O_{k}\left(n^{k/2-1+\epsilon}\right)$ for every $\epsilon>0$, it is clear that the identity \eqref{39} is valid for Re$(\nu)>-1$. The identity \eqref{39} was originally proved by A.~I.~Popov \cite[Eq.~(6)]{popov1935}.  It was also established in \cite[p.~329, Corollary 4.6]{bdkz} by a completely different method. The special case $k=2$ of this identity was obtained Hardy \cite{hardy}, who used it to prove his famous $\Omega$-theorem for the Gauss circle problem,

\section{Example 2: $\sigma_k(n)$}

Let $\sigma_k(n)$ denote the sum of the $k$th powers of the divisors of $n$, where it is assumed that $k$ is an odd positive integer.  The generating function for $\sigma_k(n)$ is given by
\begin{equation}\label{40a}
\zeta_k(s):=\zeta(s)\zeta(s-k)=\sum_{n=1}^{\infty}\df{\sigma_k(n)}{n^s}, \qquad \sigma > k+1,
\end{equation}
and it satisfies the functional equation
\begin{equation}\label{41a}
(2\pi)^{-s}\Gamma(s)\zeta_k(s)=(-1)^{(k+1)/2}(2\pi)^{-(k+1-s)}\Gamma(k+1-s)\zeta_k(k+1-s).
\end{equation}
In the notation of the  Dirichlet series and functional equation in \eqref{18} and \eqref{19}, respectively,
 $$  a(n)=\sigma_k(n), \quad b(n)=(-1)^{(k+1)/2}\sigma_k(n), \quad\lambda_n=\mu_n=n,\quad \delta=k+1.$$

Assuming that $\Re(\nu), \Re(s) >0$, apply Theorem \ref{thm3}.  First, the left-hand side of \eqref{30} is equal to
\begin{equation}\label{42a}
\df{2}{s}\sum_{n=1}^{\infty}\sigma_k(n)n^{(\nu+1)/2}K_{\nu+1}(s\sqrt{n}).
\end{equation}
The first expression on the right-hand side of \eqref{30} is readily seen to equal
\begin{equation}\label{43a}
2^{3k+\nu+4}s^{\nu}\pi^{k+1}\Gamma(k+\nu+2)\sum_{n=1}^{\infty}\df{(-1)^{(k+1)/2}\sigma_k(n)}{(s^2+16\pi^2n)^{k+\nu+2}}.
\end{equation}
It remains to evaluate the integral on the right-hand side of \eqref{30}.

Now $Q_0(s)$ is the sum of the residues of \cite[p.~17]{cn1}
\begin{equation}\label{44a}
R(z):=\df{\Gamma(z)\zeta(z)\zeta(z-k)x^{z}}{\Gamma(z+1)}.
\end{equation}
(In Chandrasekaran and Narasimhan's paper \cite{cn1}, they utilize a different convention for Bernoulli numbers, and so our representation for $Q_0$ takes a different form from theirs.)
Observe that $R(z)$ has simple poles at $z=0,-1,k+1$. Using Euler's formula,
$$\zeta(2n)=(-1)^{n-1}\df{(2\pi)^{2n}B_{2n}}{2(2n)!}, $$
where $n$ is a positive integer and $B_{n}$ denotes the $n$th Bernoulli number, we readily find that
\begin{equation}\label{45a}
Q_0(x)=\df{B_{k+1}}{2(k+1)}-\df{\delta_{1,k}x}{2}
+\df{(2\pi)^{k+1}(-1)^{(k-1)/2}B_{k+1}x^{k+1}}{2(k+1)\Gamma(k+2)},
\end{equation}
  where
  \begin{equation*}
  \delta_{1,k}=\begin{cases}
  1, \quad \text{ if } k=1,\\
  0, \quad \text{otherwise}.
  \end{cases}
  \end{equation*}
Thus,
\begin{align}
&\int_{0}^{\infty}Q_0(x)x^{\nu/2}K_{\nu}(s\sqrt{x})dx\notag\\
=&\int_{0}^{\infty}\left\{\df{B_{k+1}}{2(k+1)}-\df{\delta_{1,k}x}{2}
+\df{(2\pi)^{k+1}(-1)^{(k-1)/2}B_{k+1}x^{k+1}}{2(k+1)\Gamma(k+2)}\right\}x^{\nu/2}K_{\nu}(s\sqrt{x})dx\notag\\
=&:I_1+I_2+I_3.\label{46}
\end{align}
First, as in \eqref{33}, we find that
\begin{equation}\label{47}
I_1=\df{2^{\nu}B_{k+1}\Gamma(\nu+1)}{(k+1)s^{\nu+2}}.
\end{equation}
Secondly, with the use of Lemma \ref{Kintegral},
\begin{align}\label{48}
I_2=&-\df{\delta_{1,k}}{2}\int_0^{\infty}x^{1+\nu/2}K_{\nu}(s\sqrt{x})dx\notag\\
&=-\df{\delta_{1,k}}{s^{\nu+4}}\int_0^{\infty}t^{\nu+3}K_{\nu}(t)dt\notag\\
&=-\df{\delta_{1,k}}{s^{\nu+4}}2^{\nu+2}\Gamma(\nu+2).
\end{align}
Thirdly, employing Lemma \ref{Kintegral}, we deduce that
\begin{align}\label{49}
I_3=&\df{(2\pi)^{k+1}(-1)^{(k-1)/2}B_{k+1}}{2(k+1)\Gamma(k+2)}\int_{0}^{\infty}x^{k+1+\nu/2}K_{\nu}(s\sqrt{x})dx\notag\\
=&\df{(2\pi)^{k+1}(-1)^{(k-1)/2}B_{k+1}}{(k+1)\Gamma(k+2)s^{2k+\nu+4}}\int_{0}^{\infty}t^{2k+\nu+3}K_{\nu}(t)dt\notag\\
=&\df{2^{3k+\nu+3}\pi^{k+1}(-1)^{(k-1)/2}B_{k+1}}{(k+1)s^{2k+\nu+4}}\Gamma(k+\nu+2).
\end{align}

Finally, put \eqref{47}--\eqref{49} in \eqref{46}; then substitute \eqref{42a}, \eqref{43a}, and \eqref{46}
into Theorem \ref{thm3}; lastly multiply both sides by $s/2$ to conclude that
\begin{align}
\sum_{n=1}^{\infty}\sigma_k(n)n^{(\nu+1)/2}K_{\nu+1}(s\sqrt{n})
=&2^{3k+\nu+3}s^{\nu+1}\pi^{k+1}\Gamma(k+\nu+2)
\sum_{n=1}^{\infty}\df{(-1)^{(k+1)/2}\sigma_k(n)}{(s^2+16\pi^2n)^{k+\nu+2}}
\notag\\
&+\df{2^{{\nu}-1}B_{k+1}}{(k+1)s^{\nu+1}}\Gamma(\nu+1)
-\df{\delta_{1,k}}{s^{\nu+3}}2^{\nu+1}\Gamma(\nu+2)\notag\\
&+\df{2^{3k+\nu+2}\pi^{k+1}(-1)^{(k-1)/2}B_{k+1}}{(k+1)s^{2k+\nu+3}}\Gamma(k+\nu+2).\label{49b}
\end{align}

We now put \eqref{49b} in a more palatable form.  From \eqref{41a},
$$\zeta_k(0)=\zeta(0)\zeta(-k)=-\df12\cdot-\df{B_{k+1}}{k+1}=\df{B_{k+1}}{2(k+1)},$$
by \cite[p.~12]{edwards}.
Define
\begin{equation}\label{49cc}
\sigma_k(0)=-\zeta_k(0)=-\df{B_{k+1}}{2(k+1)}.
\end{equation}
Next, define a term for $n=0$ on the left-hand side of \eqref{49b} by formally using Lemma \ref{limit}.  Therefore, with also the use of \eqref{49cc},
\begin{align}
\lim_{n\to 0}\sigma_k(0)n^{(\nu+1)/2}K_{\nu+1}(s\sqrt{n})
=&-\df{2^{\nu-1}B_{k+1}\Gamma(\nu+1)}{(k+1)s^{\nu+1}}.\label{49c}
\end{align}
 Hence, utilizing \eqref{49c} in \eqref{49b}, we have shown that
\begin{align}
&\qquad\qquad\sum_{n=0}^{\infty}\sigma_k(n)n^{(\nu+1)/2}K_{\nu+1}(s\sqrt{n})\notag\\
=&2^{3k+\nu+3}s^{\nu+1}\pi^{k+1}\Gamma(k+\nu+2)
\sum_{n=1}^{\infty}\df{(-1)^{(k+1)/2}\sigma_k(n)}{(s^2+16\pi^2n)^{k+\nu+2}}
\notag\\
&-\df{\delta_{1,k}}{s^{\nu+3}}2^{\nu+1}\Gamma(\nu+2)
+\df{2^{3k+\nu+2}\pi^{k+1}(-1)^{(k-1)/2}B_{k+1}}{(k+1)s^{2k+\nu+3}}\Gamma(k+\nu+2).\label{49d}
\end{align}
Note that the term for $n=0$ on the right-hand side of \eqref{49d} is equal to
\begin{align}\label{49e}
&\df{2^{3k+\nu+3}s^{\nu+1}\pi^{k+1}\Gamma(k+\nu+2)(-1)^{(k+1)/2}\sigma_k(0)}{s^{2k+2\nu+4}}\notag\\
=&\df{2^{3k+\nu+2}\pi^{k+1}\Gamma(k+\nu+2)(-1)^{(k-1)/2}B_{k+1}}{(k+1)s^{2k+\nu+3}},
\end{align}
by \eqref{49cc}.  Using \eqref{49e} in \eqref{49d}, we conclude that
\begin{align*}
&\sum_{n=0}^{\infty}\sigma_k(n)n^{(\nu+1)/2}K_{\nu+1}(s\sqrt{n})
+\df{\delta_{1,k}}{s^{\nu+3}}2^{\nu+1}\Gamma(\nu+2)\notag\\
=&2^{3k+\nu+3}s^{\nu+1}\pi^{k+1}\Gamma(k+\nu+2)
\sum_{n=0}^{\infty}\df{(-1)^{(k+1)/2}\sigma_k(n)}{(s^2+16\pi^2n)^{k+\nu+2}}.
\end{align*}
Note that the elementary bound $\sigma_k(n)=O(n^{k+\epsilon})$ implies that the identity above is actually valid for Re$(\nu)>-1$. Letting $\nu=-1/2$ and using \eqref{K} leads to the following result of Chandrasekharan and Narasimhan \cite[Equation (60)]{cn1}:
\begin{align}
\sum_{n=1}^{\infty}\sigma_k(n)e^{-s\sqrt{n}}&=2^{3k+3}\Gamma\left(k+\frac{3}{2}\right)\pi^{k+\frac{1}{2}}\sum_{n=1}^{\infty}\frac{s(-1)^{(k+1)/2}\sigma_k(n)}{(s^2+16\pi^2n)^{k+3/2}}\nonumber\\
&\quad+\frac{B_{k+1}}{2(k+1)}-\frac{\delta_{1,k}}{s^2}+\frac{2^{3k+2}\pi^{k+\frac{1}{2}}(-1)^{(k-1)/2}B_{k+1}}{(k+1)s^{2k+2}}\Gamma\left(k+\frac{3}{2}\right).
\end{align}

\section{Example 3: $\tau(n)$}

Recall that the Dirichlet series for Ramanujan's arithmetical function $\tau(n)$
\begin{equation*}
f(s):=\sum_{n=1}^{\infty}\df{\tau(n)}{n^s}, \quad \sigma > \df{13}{2},
\end{equation*}
satisfies the functional equation
\begin{equation}\label{50}
\chi(s):=(2\pi)^{-s}\Gamma(s)f(s)=(2\pi)^{-(12-s)}\Gamma(12-s)f(12-s).
\end{equation}
The function $\chi(s)$ is an entire function, and so $Q_0(x)\equiv0$.  Clearly, $\lambda_n=\mu_n=n$ and $\delta=12$.  Thus, for $\Re(\nu), \Re(s)>0$, from Theorem \ref{thm3} we can immediately deduce the identity
\begin{equation}\label{51}
\sum_{n=1}^{\infty}\tau(n)n^{(\nu+1)/2}K_{\nu+1}(s\sqrt{n})=
2^{36+\nu}s^{\nu+1}\pi^{12}\Gamma(13+\nu)\sum_{n=1}^{\infty}\df{\tau(n)}{(s^2+16\pi^2n)^{\nu+13}}.
\end{equation}

Let $\nu=-\tf12$.  Then, from \eqref{K},
\begin{equation}\label{52}
K_{1/2}(s\sqrt{n})=\left(\df{\pi}{2s\sqrt{n}}\right)^{1/2}e^{-s\sqrt{n}}.
\end{equation}
Hence, by \eqref{51} and \eqref{52},
\begin{equation}\label{53}
\sum_{n=1}^{\infty}\tau(n)e^{-s\sqrt{n}}
=2^{36}\pi^{23/2}\Gamma\left(\frac{25}{2}\right)\sum_{n=1}^{\infty}\df{s\tau(n)}{(s^2+16\pi^2n)^{25/2}}.
\end{equation}
The identity \eqref{53} is originally due to Chandrasekharan and Narasimhan \cite[p.~16, Eq.~(56)]{cn1}.

\section{A theorem of G.~N.~Watson}
In this section, we obtain a theorem of Watson \cite[Equation (4)]{watsonselfreciprocal} as a special case of Theorem \ref{thm3}, namely, for Re$(z)>0$ and Re$(\nu)>0$,
\begin{align}\label{watsoni}
\frac{1}{2}\Gamma(\nu)&+2\sum_{n=1}^{\infty}\left(\frac{1}{2}nz\right)^{\nu}K_{\nu}(nz)\nonumber\\
&=\Gamma\left(\frac{1}{2}\right)\Gamma\left(\nu+\frac{1}{2}\right)z^{2\nu}\left\{\frac{1}{z^{2\nu+1}}+2\sum_{n=1}^{\infty}\frac{1}{(z^2+4\pi^2n^2)^{\nu+1/2}}\right\}.    
\end{align}
The functional equation of the Riemann zeta function is given by \cite[p.~14]{edwards}
\begin{equation}
\pi^{-s/2}\Gamma(s/2)\zeta(s)=\pi^{-(1-s)/2}\Gamma((1-s)/2)\zeta(1-s).
\end{equation}
Hence, replacing $s$ by $2s$, we see that it can be converted into the form in \eqref{19} with $\delta=1/2$. Therefore, we invoke Theorem \ref{thm3} with $a(n)=b(n)=1, \lambda_n=\mu_n=n^2/2$, whence
\begin{align}\label{w1}
\frac{2}{s}\sum_{n=1}^{\infty}(n/\sqrt{2})^{\nu+1}K_{\nu+1}(sn/\sqrt{2})&=2^{\nu+5/2}s^{\nu}\pi^{1/2}\Gamma(\nu+3/2)\sum_{n=1}^{\infty}\frac{1}{(s^2+8\pi^2n^2)^{\nu+3/2}}\nonumber\\
&\quad+\int_{0}^{\infty}Q_0(x)x^{\nu/2}K_{\nu}(s\sqrt{x})\, dx.
\end{align}
Here $Q_0(x)=-1/2+\sqrt{2x}$, the sum of the residues of $\zeta(2z)(2x)^{z}/z$ at $0$ and $1/2$, so that 
\begin{align}\label{w2}
\int_{0}^{\infty}Q_0(x)x^{\nu/2}K_{\nu}(s\sqrt{x})\, dx=2^{\nu+3/2}\sqrt{\pi}s^{-\nu-3}\Gamma(\nu+3/2)-2^{\nu}s^{-\nu-2}\Gamma(1+\nu).    
\end{align}
From \eqref{w1} and \eqref{w2},
\begin{align}
\frac{2^{(1-\nu)/2}}{s}\sum_{n=1}^{\infty}(n/\sqrt{2})^{\nu+1}K_{\nu+1}(sn/\sqrt{2})&=2^{\nu+5/2}s^{\nu}\pi^{1/2}\Gamma(\nu+3/2)\sum_{n=1}^{\infty}\frac{1}{(s^2+8\pi^2n^2)^{\nu+3/2}}\nonumber\\
&\quad+2^{\nu+3/2}\sqrt{\pi}s^{-\nu-3}\Gamma(\nu+3/2)-2^{\nu}s^{-\nu-2}\Gamma(1+\nu). 
\end{align}
Next let $s=z\sqrt{2}$ in the foregoing equation, then multiply the resulting identity by $2^{-\nu/2}z^{\nu+2}$, and replace $\nu$ by $\nu-1$ to arrive at \eqref{watsoni} upon simplification.
\section{Primitive characters $\chi(n)$}
Let $\chi$ denote a primitive character modulo $q$.  Because the functional equations for the Dirichlet $L$-series
$$L(s,\chi)=\sum_{n=1}^{\infty}\df{\chi(n)}{n^s}, \qquad \sigma >0,$$
are different for $\chi$ even and $\chi$ odd, we separate the two cases.

Suppose first that $\chi$ is odd.  Then the functional equation for $L(s,\chi)$ is given by \cite[p.~71]{davenport}
\begin{equation}\label{funcequa1}
\chi(s):=\left(\df{\pi}{q}\right)^{-s}\Gamma(s)L(2s-1,\chi)=
-\df{i\tau(\chi)}{\sqrt{q}}\left(\df{\pi}{q}\right)^{-(\tf32-s)}\Gamma\left(\tf32-s\right)L(2-2s,\overline{\chi}),
\end{equation}
where $\overline{\chi}(n)$ denotes the complex conjugate of $\chi(n)$, and $\tau(\chi)$ denotes the Gauss sum
$$\tau(\chi):=\sum_{n=1}^{q}\chi(n)e^{2\pi in/q}.$$
Hence, in the notation of \eqref{18} and \eqref{19},
$$a(n)=n\chi(n),\quad b(n)=-\df{i\tau(\chi)}{\sqrt{q}}n\overline{\chi}(n), \quad \lambda_n=\mu_n=\df{n^2}{2q},
\quad \delta=\df32.$$
 Also, $\chi(s)$ is an entire function, and consequently $Q_0(x)\equiv0$.  Hence, by Theorem \ref{thm3}, for $\Re(\nu), \Re(s)>0$,
\begin{gather}
\df{2}{s}\sum_{n=1}^{\infty}n\chi(n)
\left(\df{n^2}{2q}\right)^{(\nu+1)/2}K_{\nu+1}\left(s\sqrt{\df{1}{2q}}n\right)\notag\\
=-\df{i\tau(\chi)}{\sqrt{q}}2^{\nu+11/2}s^{\nu}\pi^{3/2}\Gamma\left(\nu+\tfrac52\right)
\sum_{n=1}^{\infty}\df{n\overline{\chi}(n)}{\left(s^2+8\pi^2n^2/q\right)^{\nu+5/2}}.\label{54}
\end{gather}
Now multiply both sides of \eqref{54} by
$\frac{1}{2}s\left(2q\right)^{(\nu+1)/2}$, and then let
$s=\sqrt{2q}\,r$
to deduce that
\begin{equation}
\sum_{n=1}^{\infty}\chi(n)
n^{\nu+2}K_{\nu+1}(rn)
=-i\tau(\chi)
\df{r^{\nu+1}q^{2\nu+3}\Gamma\left(\nu+\tfrac52\right)}{2^{\nu+2}\pi^{2\nu+7/2}}
\sum_{n=1}^{\infty}\df{n\overline{\chi}(n)}{\left(n^2+q^2r^2/(4\pi^2)\right)^{\nu+5/2}}.\label{55a}
\end{equation}

Second, let $\chi$ be even.  Then the functional equation of $L(s,\chi)$ is given by \cite[p.~69]{davenport}
 \begin{equation}\label{funcequa2}
\chi(s):=\left(\df{\pi}{q}\right)^{-s}\Gamma(s)L(2s,\chi)=
\df{\tau(\chi)}{\sqrt{q}}\left(\df{\pi}{q}\right)^{-(\tf12-s)}\Gamma\left(\tf12-s\right)L(1-2s,\overline{\chi}),
\end{equation}
Hence, by \eqref{18} and \eqref{19},
$$a(n)=\chi(n),\quad b(n)=\df{\tau(\chi)}{\sqrt{q}}\overline{\chi}(n), \quad \lambda_n=\mu_n=\df{n^2}{2q},
\quad \delta=\df12.$$
Also, $\chi(s)$ is an entire function, and consequently $Q_0(x)\equiv0$.  Hence, by Theorem \ref{thm3}, for $\Re(\nu), \Re(s)>0$,
\begin{gather}
\df{2}{s}\sum_{n=1}^{\infty}\chi(n)
\left(\df{n^2}{2q}\right)^{(\nu+1)/2}K_{\nu+1}\left(s\sqrt{\df{1}{2q}}n\right)\notag\\
=\df{\tau(\chi)}{\sqrt{q}}2^{\nu+5/2}s^{\nu}\pi^{1/2}\Gamma\left(\nu+\tfrac32\right)
\sum_{n=1}^{\infty}\df{\overline{\chi}(n)}{\left(s^2+8\pi^2n^2/q\right)^{\nu+3/2}}.\label{56}
\end{gather}
Multiply both sides of \eqref{54} by $\tf12 s(2q)^{(\nu+1)/2}$ and then let $s=\sqrt{2q}\,r$
to obtain
\begin{gather}
\sum_{n=1}^{\infty}\chi(n)n^{\nu+1} K_{\nu+1}(rn)
=\tau(\chi)
\df{r^{\nu+1}q^{2(\nu+1)}\Gamma\left(\nu+\tfrac32\right)}{2^{\nu+2}\pi^{2\nu+5/2}}
\sum_{n=1}^{\infty}\df{\overline{\chi}(n)}{\left(n^2+q^2r^2/(4\pi^2)\right)^{\nu+3/2}}.\label{57a}
\end{gather}

Identities \eqref{55a} and \eqref{57a} were first obtained in \cite[Theorem 2.1]{bds} and are character analogues of \eqref{watsoni}.

\section{Ideal Functions $F(n)$ of Imaginary Quadratic Fields}

Let $F(n)$ denote the number of integral ideals of norm $n$ in an imaginary quadratic number  field $K=\mathbb{Q}\left(\sqrt{-D}\right)$, where $D$ is the discriminant of $K$.  Then the Dedekind zeta function 
$$\zeta_{K}(s):=\sum_{n=1}^{\infty}\df{F(n)}{n^s}, \qquad \sigma>1,$$
satisfies the functional equation
\begin{equation}\label{dedekindk}
\left(\df{2\pi}{\sqrt{D}}\right)^{-s}\Gamma(s)\zeta_K(s)
=\left(\df{2\pi}{\sqrt{D}}\right)^{s-1}\Gamma(1-s)\zeta_K(1-s).
\end{equation}
We note from \eqref{18} and \eqref{19} that
$$ a(n)=b(n)=F(n), \qquad \lambda_n=\mu_n=n/\sqrt{D}, \qquad \delta=1.$$
The function $\zeta_K(s)$ has an analytic continuation to the entire complex plane where it is analytic except for a simple pole at $s=1$.  From \cite[p.~212]{cohen},
\begin{equation}\label{pole}
\lim_{s\to 1}(s-1)\zeta_K(s)=\df{2\pi h(K)R(K)}{w(K)\sqrt{D}},
\end{equation}
where $h(K), R(K)$, and $w(K)$ denote, respectively, the class number of $K$, the regulator of $K$, and the number of roots of unity in $K$. 
Furthermore, from \eqref{dedekindk} and \eqref{pole},
\begin{equation}\label{zerovalue}
\zeta_K(0)=\lim_{s\to 0}\df{\sqrt{D}}{2\pi}\cdot\df{1}{s\Gamma(s)}\cdot s\zeta_K(1-s)
=\df{\sqrt{D}}{2\pi}\cdot-\df{2\pi h(K)R(K)}{w(K)\sqrt{D}}=-\df{h(K)R(K)}{w(K)}.
\end{equation}
For simplicity, set $d=\sqrt{D},  h=h(K), R=R(K)$, and $w=w(K)$.
From \eqref{zerovalue} and \eqref{pole},
\begin{align}
Q_0(x)=&\df{1}{2\pi i}\int_{\mathcal{C}}\df{\Gamma(z)}{\Gamma(z+1)}d^z
\zeta_{K}(z)x^zdz\notag\\
=&-\df{hR}{w}+\df{2\pi hRx}{w}.\label{70}
\end{align}

We apply Theorem \ref{thm3}.  First, we calculate the integral on the right-hand side of \eqref{30}. To that end, by \eqref{70} and \eqref{zerovalue},
\begin{align}\label{60}
I:=&\int_0^{\infty}Q_0(x)x^{\nu/2}K_{\nu}(s\sqrt{x})dx\notag\\
=&-\df{hR}{w}\int_0^{\infty}x^{\nu/2}K_{\nu}(s\sqrt{x})dx
+\df{2\pi hR}{w}\int_0^{\infty}x^{\nu/2+1}K_{\nu}(s\sqrt{x})dx
\notag\\
=&-\df{hR}{w}\df{2^{\nu+1}}{s^{\nu+2}}\Gamma(\nu+1)
+\df{2\pi hR}{w}\df{2^{\nu+3}}{s^{\nu+4}}\Gamma(\nu+2),
\end{align}
by Lemmas \ref{23} and \ref{Kintegral}.  Hence, with the use of \eqref{60}, Theorem \ref{thm3} yields
\begin{align}
\sum_{n=1}^{\infty}F(n)(n/d)^{(\nu+1)/2}&K_{\nu+1}(s\sqrt{n/d})
=2^{\nu+3}s^{\nu+1}\pi\Gamma(\nu+2)\sum_{n=1}^{\infty}\df{F(n)}{(s^2+16\pi^2n/d)^{\nu+2}}\notag\\
&-\df{hR}{w}\df{2^{\nu}}{s^{\nu+1}}\Gamma(\nu+1)
+\df{2\pi hR}{wd}\df{2^{\nu+2}}{s^{\nu+3}}\Gamma(\nu+2).
\label{61}
\end{align}
Let $s=4\pi\sqrt{r/d}$ and multiply both sides by $d^{(\nu+1)/2}$.  Hence,
\begin{align}
&\sum_{n=1}^{\infty}F(n)n^{(\nu+1)/2}K_{\nu+1}(4\pi \sqrt{rn}/d)
=\df{1}{2\sqrt{r}}\left(\frac{d\sqrt{r}}{2\pi}\right)^{\nu+2}\Gamma(\nu+2)\sum_{n=1}^{\infty}\df{F(n)}{(r+n)^{\nu+2}}\notag\\
&-\df{hR}{2w}\left(\df{d}{2\pi\sqrt{r}}\right)^{\nu+1}\Gamma(\nu+1)
+\df{hR}{2w\sqrt{r}}\left(\df{d}{2\pi\sqrt{r}}\right)^{\nu+2}\Gamma(\nu+2).\label{63}
\end{align}

From a formal use of Lemma \ref{limit},
\begin{equation}\label{64}
\lim_{n\to 0}n^{(\nu+1)/2}K_{\nu+1}\left(\frac{4\pi \sqrt{rn}}{d}\right)=\df{1}{2}\left(\frac{d}{2\pi \sqrt{r}}\right)^{\nu+1}\Gamma(\nu+1).
\end{equation}
Hence, if we define $F(0)=hR/w$ and then note that $F(0)$ multiplied by the right side of \eqref{64} appears on the right side of \eqref{63}, we can rewrite \eqref{63} in the form
\begin{align*}
\sum_{n=0}^{\infty}F(n)n^{(\nu+1)/2}K_{\nu+1}(4\pi\sqrt{rn}/d)
=&\df{1}{2\sqrt{r}}\left(\frac{d\sqrt{r}}{2\pi}\right)^{\nu+2}\Gamma(\nu+2)\sum_{n=0}^{\infty}
\df{F(n)}{(r+n)^{\nu+2}}.
\end{align*}
From \cite[Lemma 9]{cmapproximate}, we see that $F(n)=O(n^{\epsilon})$ for every $\epsilon>0$.  Hence the foregoing identity is actually valid for Re$(\nu)>-1$. Hence, letting $\nu=-1/2$, we obtain the special case 
\begin{equation*}
\sum_{n=0}^{\infty}F(n)e^{-4\pi\sqrt{rn}/d}=\frac{d^{3/2}\sqrt{r}}{4\pi}\sum_{n=0}^{\infty}
\df{F(n)}{(r+n)^{3/2}}.
\end{equation*}

\begin{center}
\textbf{Acknowledgements}
\end{center}

The first and second authors sincerely thank the MHRD SPARC project SPARC/2018-2019/P567/SL for the financial support.

\end{document}